\documentclass[11pt]{amsart}
\usepackage{a4wide}
\usepackage[latin1]{inputenc}
\usepackage{amsmath}
\usepackage{amsfonts}
\usepackage{amssymb}
\usepackage{graphicx}
\usepackage{amsthm}
\usepackage{amssymb}
\usepackage{cite}
\usepackage{mathrsfs}
\usepackage{hyperref}
\usepackage{paralist}
\usepackage{enumitem}
\usepackage{tikz}
\usetikzlibrary{matrix,chains}

\newcommand{\B}{\mathbf{B}}
\newcommand{\U}{\mathbf{U}}
\newcommand{\T}{\mathbf{T}}
\newcommand{\G}{\mathbf{G}}
\newcommand{\Para}{\mathbf{P}}
\newcommand{\Levi}{\mathbf{L}}
\newcommand{\Y}{\mathbf{Y}}

\newcommand{\Gtilde}{\mathbf{\tilde{G}}}

\newcommand{\Ltilde}{\mathbf{\tilde{L}}}
\newcommand{\C}{\operatorname{C}}
\newcommand{\N}{\mathbf{N}}
\newcommand{\n}{\mathbf{n}}
\newcommand{\x}{\mathbf{x}}
\newcommand{\h}{\mathbf{h}}
\newcommand{\m}{\mathbf{m}}
\newcommand{\Z}{\operatorname{Z}}

\newcommand{\End}{\operatorname{End}}
\newcommand{\Ind}{\operatorname{Ind}}
\newcommand{\Res}{\operatorname{Res}}

\theoremstyle{remark}

\theoremstyle{definition}
\newtheorem{definition}{Definition}

\newtheorem{remark}[definition]{Remark}

\theoremstyle{plain}
\newtheorem{theorem}[definition]{Theorem}

\newtheorem{lemma}[definition]{Lemma}

\newtheorem{proposition}[definition]{Proposition}

\newtheorem{assumption}[definition]{Assumption}

\makeindex

\begin{document}

\title{On the Bonnafé--Dat--Rouquier Morita equivalence}
\author{Lucas Ruhstorfer}
\date{\today}

\address{School of Mathematics and Natural Sciences University of Wuppertal, Gau\ss str. 20, 42119 Wuppertal, Germany}
%\address{Fachbereich Mathematik, Bergische Universität Wuppertal, Gaußstraße 20, 42119 Wuppertal, Germany}
\email{ruhstorfer@uni-wuppertal.de}


\maketitle

\begin{abstract}
We prove that the cohomology group of a Deligne--Lusztig variety defines a Morita equivalence in a case which is not covered by the argument in \cite{Dat}, specifically we consider the situation for semisimple elements in type $D$ whose centralizer has non-cyclic component group. Some arguments use considerations already present in an unpublished note by Bonnafé, Dat and Rouquier. \end{abstract} 

\section*{Introduction}
Let $\G$ be a connected reductive group with Frobenius endomorphism $F: \G \to \G$ defining an $\mathbb{F}_q$-structure, where $q$ is a power of a prime $p$. Let $\G^\ast$ be a group dual to $\G$ with dual Frobenius $F^\ast: \G^\ast \to \G^\ast$. Let $\ell$ be a prime number different from $p$ and $(\mathcal{O},K,k)$ an $\ell$-modular system as in \cite[Section 2.A]{Dat}. Let $s\in (\G^\ast)^{F^\ast}$ be a semisimple element of $\ell'$-order and $\Levi^\ast$ be the minimal $F^\ast$-stable Levi subgroup containing $\C^\circ_{\G^\ast}(s)$ and $\Levi$ be the Levi subgroup of $\G$ dual to $\Levi^\ast$. In addition, let $e_s^{\Levi^F} \in \mathrm{Z}(\mathcal{O} \Levi^F)$ be the central idempotent associated to $s\in (\Levi^\ast)^{F^\ast}$, see \cite[Theorem 9.12]{MarcBook}, and $\mathbf{N}^F$ denotes $\mathrm{N}_{\G^F}(\Levi^F, e_s^{\Levi^F})$.

Let $\Para= \Levi \U$ be a Levi decomposition in $\G$ and denote by $\Y_\U^\G$ the associated Deligne--Lusztig variety on which $ \G^F$ acts on the left and $ \Levi^F$ acts on the right. Denote $d:= \mathrm{dim}(\Y_\U^\G)$ and let $H^d_c(\Y_\U^\G,\mathcal{O})$ be the $d$th $\ell$-adic cohomology group with compact support of $\Y_\U^\G$. Suppose that the $\mathcal{O} \G^F$-$\mathcal{O}\Levi^F$-bimodule $H_c^d(\Y_\U^\G,\mathcal{O}) e_s^{\Levi^F}$ extends to an $\mathcal{O} \G^F$-$\mathcal{O}\N^F$-bimodule. Then in \cite[Theorem 7.7]{Dat} the authors prove the astonishing result that this extended bimodule induces a Morita equivalence between $\mathcal{O} \mathbf{N}^F e_s^{\Levi^F}$ and $\mathcal{O} \G^F e_s^{\G^F}$. This strengthens an earlier theorem of Bonnafé and Rouquier proving a conjecture of Broué \cite{Broue}.

Note however that \cite[Theorem 7.7]{Dat} was announced without the assumption that this bimodule extends. As the proof of \cite[Proposition 7.3]{Dat} is incomplete, this assumption is necessary at the moment. One case where this assumption is easily seen to be satisfied is when the quotient subgroup $\mathbf{N}^F/ \Levi^F$ is cyclic, see Lemma \ref{assumption} below.

Our aim in this note is to remove this technical assumption and therefore extend the results of \cite[Theorem 7.7]{Dat}. From now on we assume that $\G$ is a simple algebraic group. In this case, the quotient group $\mathbf{N}^F/\Levi^F$ embeds into $\mathrm{Z}(\G)^F$. Therefore, a non-cyclic quotient can only appear if $\G$ is simply connected and $\G^F$ is of type $D_n$ with even $n\geq 4$. Hence we focus on this situation and prove the following statement.

\begin{proposition}\label{ext}
Let $\G$ be a simple, simply connected algebraic group such that $\G^F$ is of type $D_n$ with even $n \geq 4$. If $\ell \nmid (q^2-1)$ then the $ \mathcal{O} \G^F$-$\mathcal{O} \Levi^F$-bimodule $H_c^d(\Y_\U^\G,\mathcal{O}) e_s^{\Levi^F}$ extends to an $\mathcal{O} \G^F$-$\mathcal{O} \N^F$-bimodule.
\end{proposition}

The proof combines group theoretic descriptions of the relevant Levi subgroups and Clifford theoretic arguments tailored to this situation. Unfortunately, the restriction on $\ell$ seems to be necessary with the approach presented, see proof of Proposition \ref{main}.

Using an argument from the unpublished note \cite{Note} we show that the extended bimodule induces a Morita equivalence and from this we can deduce the validity of \cite[Theorem 7.7]{Dat} in this case.

\begin{theorem}\label{final}
Suppose that $\G$ is a simple algebraic group. If $ \ell \nmid (q^2-1)$ or if $\mathbf{N}^F/ \Levi^F$ is cyclic then the complex $G \Gamma_c(\Y_\U^\G,\mathcal{O})^{\mathrm{red}} e_s^{\Levi^F}$ of $\mathcal{O} \G^F$-$\mathcal{O} \Levi^F$-bimodules extends to a complex $C$ of $\mathcal{O} \G^F$-$\mathcal{O} \mathbf{N}^F$-bimodules. There exists a unique bijection $b \mapsto b'$ between blocks of $\mathcal{O} \G^F e_s^{\G^F}$ and blocks of $\mathcal{O} \mathbf{N}^F e_s^{\Levi^F}$ such that $ b C \cong C b'$. For each block $b$ the complex $b C b'$ induces a splendid Rickard equivalence between $b$ and $b'$. Similarily, the bimodule $H^d(b C b')$ induces a Morita equivalence between $b$ and $b'$. The complex $bC b'$ induces an isomorphism of the Brauer categories of $k \G^F b$ and $k \G^F b'$. In particular, $b$ and $b'$ have the same defect group.
\end{theorem}

\section*{Acknowledgement}

Firstly, I would like to thank Cedric Bonnafé and Raphael Rouquier for insightful discussions.
I also thank Marc Cabanes and Gunter Malle for helpful comments on a previous version of this paper.
Finally, Britta Späth for her continuous support.

This material is partly based upon work supported by the NSF under Grant DMS-1440140 while the author was in residence at the MSRI, Berkeley CA. This research was conducted in the framework of the research training group GRK 2240: Algebro-geometric
Methods in Algebra, Arithmetic and Topology, which is funded by the DFG.

\section*{Notation}

We introduce the notation which will be in force until the last section of this paper. Let $\G^\ast$ be a simple, adjoint algebraic group of type $D_{n}$ with $n$ even and $F^{\ast}: \G \to \G$ be a Frobenius endomorphism defining an $\mathbb{F}_q$-structure on $\G^\ast$ such that $\G^{F^\ast}$ is of untwisted type $D_n$. Fix a semisimple element $s \in (\G^\ast)^{F^\ast}$. Then $ \C_{\G^\ast}^\circ(s)$ is an $F^\ast$-stable connected reductive group. Thus, there exists a maximal $F^\ast$-stable torus $\T_0^\ast$ of $\C_{\G^\ast}^\circ(s)$ contained in an $F^\ast$-stable Borel subgroup $\B(s)$ of $\C_{\G^\ast}^\circ(s)$. 

As the dual group $\G$ is of simply connected type, there exists a surjective morphism $\pi: \G \to \G^\ast$ with kernel $\Z(\G)$. We let $\T_0$ be the maximal torus of $\G$ such that $\T_0^\ast= \pi(\T_0)$. Let $F: \G \to \G$ be a Frobenius endomorphism stabilizing $\T_0$ such that $ (\G, \T_0, F)$ is in duality with $ (\G^\ast, \T_0^\ast,F^\ast)$ via the the map $ \pi: \T_0 \to \T_0^\ast$.

We denote by $W$ the Weyl group of $\G$ with respect to $\T_0$ and by $W^\ast$ the Weyl group of $\G^\ast$ with respect to $\T_0^\ast$. The map $\pi$ induces an isomorphism $W \to W^\ast$ which allows $W$ to be identified with $W^\ast$. Under this identification, the anti-isomorphism ${}^\ast: W \to W^\ast$, induced by duality, is then given by inversion, i.e. $w^\ast=w^{-1}$ for all $w \in W$.

The root system of $\G$ can be described more explicitly as follows. Let $\overline{\Phi}$ be a root system of type $B_n$, $n$ even, with base $\{e_1,e_i-e_{i-1}  \mid 2 \leq i \leq n \}$ where $\{ e_i \mid 1 \leq i \leq n \}$ is the canonical orthonormal basis with respect to the standard scalar product on $\mathbb{R}^n$.

Consider the root system $\Phi \subseteq \overline{\Phi}$ consisting of all long roots of $\overline{\Phi}$. Recall that $\Phi$ is a root system of type $D_n$. Let $\overline{\G}$ be the associated simple, simply connected algebraic group defined over $\overline{\mathbb{F}_q}$. By \cite[Section 2.C]{MS} there exists an embedding $\G \hookrightarrow \overline{\G}$ such that the image of $ \T_0$ is a maximal torus of $ \overline{\G}$. In particular, we can identify $\Phi$ with the root system of $\G$ with respect to the torus $\T_0$ and $\overline{\Phi}$ with the root system of $\overline{\G}$ with respect to $\T_0$.

Let $\x_{\bar{\alpha}}(r),\n_{\bar{\alpha}}(r)$ and $\h_{\bar{\alpha}}(r)$ ($r\in \overline{\mathbb{F}_q}$ and $\bar{\alpha} \in \overline{\Phi}$) the Chevalley generators associated to the maximal torus $\overline{\T_0}$ of $\overline{\G}$ as in \cite[Theorem 1.12.1]{GLS}.

Using the embedding of $\G$ into $\overline{\G}$ we obtain a surjective group homomorphism
$$(\overline{\mathbb{F}_q}^\times)^n \to \T_0, \, (\lambda_1,\dots, \lambda_n) \mapsto \prod_{i=1}^n \h_{e_i}(\lambda_i)$$
with kernel $\{ (\lambda_1,\dots, \lambda_n) \in \{ \pm 1\}^n \mid \prod_{i=1}^n \lambda_i=1 \}$. Hence we can write an element $\lambda \in \T_0$ (in a non-unique way) as $\lambda= \prod_{i=1}^{n} \h_{e_i}(\lambda_i)$ for suitable $\lambda_i \in \overline{\mathbb{F}_q}^\times$. For a subset $\mathcal{A} \subset \overline{\mathbb{F}_q}^\times$ with $\mathcal{A}=-\mathcal{A}$ we define 
$$I_{\mathcal{A}}(\lambda):= \{ j \in \{1, \dots, n \} \mid \lambda_j \in \mathcal{A} \}.$$
Note that this does not depend on the choice of the sequence $(\lambda_1,\dots,\lambda_n)$ but only on the element $\lambda \in \T_0$. Let $\omega_4 \in \overline{\mathbb{F}_q}^\times$ be a primitive $4$th root of unity. By \cite[Section 2.C]{MS} we have $\Z(\G)=\langle z_1, z_2 \rangle$, where $z_1=\h_{e_1}(-1)$ and $z_2=\prod_{i=1}^n  \h_{e_i}(\omega_4)$. 

We also fix a tuple $(t_1,\dots,t_n) \in (\overline{\mathbb{F}_q}^\times)^n$ such that $t= \prod_{i=1}^n \h_{e_i}(t_i)\in \T_0$ satisfies $\pi(t)=s$.

Recall that the Weyl group $\overline{W}=\N_{\overline{\G}}(\T_0)/\T_0$ can be identified with the subgroup 
$$\{ \sigma \in  S_{ \{ \pm 1, \dots, \pm n \} } \mid \sigma(-i)=-\sigma(i)  \text{ for all } i=1, \dots,n  \}$$ of $S_{ \{ \pm 1, \dots, \pm n \} }$. By \cite[Proposition 1.4.10]{GeckPfeiffer} it follows that the natural map $W \hookrightarrow \overline{W}$ identifies the Weyl group $W$ as the kernel of the group homomorphism
$$\varepsilon: \overline{W} \to \{ \pm 1 \}, \, \sigma \mapsto (-1)^{|\{i \in \{1,\dots, n \} \mid \sigma(i) < 0 \}|}.$$

Let $F_0 : \G \to \G$ be the Frobenius endomorphism defined by $\x_{\alpha}(t) \mapsto \x_{\alpha}(t^q)$, for $t \in \overline{\mathbb{F}_q}$ and $\alpha \in \Phi$. We let $F_0^\ast: \G^\ast \to \G^\ast$ be defined as the unique morphism satisfying $\pi \circ F_0=F_0^\ast \circ \pi$. Then the triple $(\G^\ast,\T_0^\ast,F_0)$ is in duality with $(\G,\T_0,F_0)$. There exists an element $v \in W$ with preimage $\m_v \in \mathrm{N}_\G(\T_0)$ of $v$ such that $F=\m_v F_0$. Since $(\G,\T_0,F)$ is in duality with $(\G,\T_0,F^\ast)$ there exists some $\m_{v^\ast} \mathrm{N}_{\G^\ast}(\T_0^\ast)$, a preimage of $v^\ast$ in $\mathrm{N}_{\G^\ast}(\T_0^\ast)$, such that $F^\ast=F_0^\ast \m_{v^\ast}$.

\section*{Classifying semisimple conjugacy classes}

Let $\Levi^\ast=\C_{\G^\ast}(\mathrm{Z}^\circ(\C^\circ_{\G^\ast}(s)))$ be the minimal Levi subgroup of $\G^\ast$ containing $\mathrm{C}^\circ_{\G^\ast}(s)$ and $\mathbf{N}^\ast= \mathrm{C}_{\G^\ast}(s) \Levi^\ast$. Let $\Levi$ be a Levi subgroup of $\G$ containing the maximal torus $\T_0$ which is in duality with $\Levi^\ast$. We set $\mathbf{N}$ be the subgroup of $\mathrm{N}_{\G}(\Levi)$ such that $\mathbf{N}/\Levi \cong \mathbf{N}^\ast/ \Levi^\ast$ under the canonical isomorphism between $\mathrm{N}_{\G}(\Levi)/ \Levi $ and $\mathrm{N}_{\G^\ast}(\Levi^\ast)/ \Levi^\ast$ induced by duality. We start by recalling the observation already made in the introduction:

\begin{lemma}\label{assumption}
In order to prove Proposition \ref{ext} we can assume that $\mathbf{N}^F/ \Levi^F $ is non-cyclic.
\end{lemma}

\begin{proof}
If $\mathbf{N}^F/ \Levi^F$ is cyclic then for instance \cite[Lemma 10.2.13]{Rouquier3} shows that $H_c^d(\Y_\U^\G,\mathcal{O}) e_s^{\Levi^F}$ extends to an $\mathcal{O} \G^F$-$\mathcal{O}\N^F$-bimodule.
\end{proof}

We will now give a more explicit description of the quotient group $\mathbf{N}^F/ \Levi^ F$. By definition we have an injective morphism 
$$\mathbf{N}^\ast/ \Levi^\ast \hookrightarrow \mathrm{C}_{\G}^\ast(s)/\mathrm{C}^\circ_{\G^\ast}(s).$$
As in \cite[Lemma 2.6]{Bonnafe} we consider the morphism 
$$\omega_s: \mathrm{C}_{\G}^\ast(s) \to \mathrm{Z}(\G), \, x \mapsto [y,t]$$
for some $y \in \G$ with $\pi(y)=x$, which by \cite[Corollary 2.8]{Bonnafe}, induces an injection 
$$\mathrm{C}_{\G}^\ast(s)/\mathrm{C}^\circ_{\G^\ast}(s) \hookrightarrow \mathrm{Z}(\G).$$
Thus, we have an embedding $\mathbf{N}/ \Levi \hookrightarrow \mathrm{Z}(\G),$
which induces a map $\mathbf{N}^F/ \Levi^F \hookrightarrow \mathrm{Z}(\G)^F$ on fixed points. As $\mathrm{Z}(\G)^F \cong C^2_{\mathrm{gcd}(2,q-1)}$ we can assume by Lemma \ref{assumption} that $q$ is odd and that $\mathbf{N}^F/ \Levi^F \cong \mathrm{Z}(\G)^F$. Let $W(s)$ (resp. $W^\circ(s)$) be the Weyl group of $\mathrm{C}_{\G^\ast}(s)$ (resp. $\mathrm{C}^\circ_{\G^\ast}(s)$) with respect to $\T_0^\ast$.
By \cite[Remark 2.4]{DM} we have a canonical isomorphism
$$W(s)/W^\circ(s) \to \mathrm{C}_{\G^\ast}(s)/ \mathrm{C}^\circ_{\G^\ast}(s).$$
Recall that $ \T_0$ is contained in a maximal $F$-stable Borel subgroup $\B(s)$ of $ \C_{\G^\ast}^\circ(s)$. Let $\Phi(s)$ be the root system of $ \C_{\G^\ast}^\circ(s)$ with set of positive roots $\Phi^+(s)$ associated to this choice. According to \cite[Proposition 1.3]{Bonnafe} we have $W(s)= A(s) \rtimes W^\circ(s)$, where
$A(s)=\{w \in W(s) \mid w(\Phi^+(s))=\Phi^+(s) \}$. Since $A(s)$ is $F$-stable this shows that the map
$$W(s)^F/W^\circ(s)^{F} \to (\mathrm{C}_{\G\ast}(s))^{F^\ast}/ (\mathrm{C}^\circ_{\G^\ast}(s))^{F^\ast}
$$
is again an isomorphism. As the morphism $\omega_s$ induces an isomorphism 
$$(\mathrm{C}_{\G^\ast}(s)/\mathrm{C}_{\G^\ast}^\circ(s))^{F^\ast} \cong \mathrm{Z}(\G)^{F} \cong C_2 \times C_2$$
we conclude that there exist $w_1^\ast,w_2^\ast \in W^{F^\ast}$ with ${}^{w_1} t=t z_1$ and ${}^{w_2} t=t z_2$. Since $W^{F^\ast}= \mathrm{C}_{W^\ast}(v)$ we have $w_1,w_2 \in \mathrm{C}_W(v)$.

\begin{remark}\label{empty}
The set $I_{ \{ \pm 1, \pm \omega_4 \} }(t)$ is non-empty.
\end{remark}

\begin{proof}
Suppose that $I_{ \{ \pm 1, \pm \omega_4 \} }(t) = \emptyset$. Write ${}^{w_1} t= \prod_{i=1}^n \h_{e_i}(s_i)$ for suitable $s_i \in \overline{\mathbb{F}_q}^\times$. Then ${}^{w_1} t t^{-1} \h_{e_1}(-1)=1$ implies that $s_i t_i^{-1} \in \{ \pm 1 \}$ for all $i$. Now note that 
$$\emptyset=I_{ \{ \pm 1, \pm \omega_4 \} }(t)= I_{ \{\pm 1, \pm \omega_4 \}}(t z_1)=I_{ \{ \pm 1, \pm \omega_4 \}}({}^{w_1} t).$$
Thus, $s_i,t_i \notin \{ \pm 1, \pm \omega_4 \}$ and so $s_i=t_i$ for all $i$. This leads to the contradiction ${}^{w_1} t=t$.
\end{proof}

\begin{lemma}\label{orbit}
In order to prove Proposition \ref{ext} we may assume that $t$ is of the form $t= \prod_{i=1}^n \h_{e_i}(t_i)$ such that $t_i=t_j$ whenever $t_j \in \{ \pm t_i, \pm t_i^{-1} \}$. 
\end{lemma}

\begin{proof}
Let $\underline{n}:=\{1,\dots,n \}$. We define the equivalence relation $\sim$ on $\underline{n}$ by saying that $i \sim j$ if $t_j \in \{ \pm t_i, \pm t_i^{-1} \}$. Let $K$ be a set of representatives for the equivalence classes of $\underline{n}$ under $\sim$.

Let $x \in I_{ \{ \pm 1, \pm \omega_4 \} }(t)$. Under the identification of the Weyl group we set
$$w:=(x,-x)^{|K|} \prod_{k \in K } (k,-k) \in W.$$
Since $\h_{e_i}(-1)=\h_{e_1}(-1)=z_1$ for all $i \in \{1, \dots, n \}$ we see that either ${}^w t$ or ${}^w t z_1$ is of the desired form. We let $t' \in \{{}^w t, {}^w t z_1 \}$ be said element. In order to prove Proposition \ref{ext} it is therefore harmless to replace $s$ by the $\G^F$-conjugate $s':={}^w s \in \T_0$. Since $\pi(t')=s'$ this element has a preimage $t' \in \T_0$ which is of the form as announced in the lemma.
\end{proof}

From now on we assume that the element $t$ has the form given in Lemma \ref{orbit}. Recall that 
$$\mathrm{C}_{\G}(t)= \langle \T_0, \x_\alpha(r) \mid \alpha \in \Phi \text{ with }\alpha(t)=1, r\in \overline{\mathbb{F}_q}^\times \rangle.$$
Let $\alpha=e_i \pm e_j \in \Phi$ with $\alpha(t)=1$. Then $ \alpha(t)=(t_i t_j^{ \pm 1})^2=1$ and therefore $t_i= \varepsilon t^{\mp 1}_j$ for some $\varepsilon \in \{ \pm 1 \}$. By assumption on $t$, this implies $t_i=t_j$. In addition, we have $\alpha=e_i-e_j$ if $t_i$ is not a $4$th root of unity. Therefore, the root system $\Phi(t)$ of $\mathrm{C}_{\G}(t)$ is given by 
$$\Phi(t)= \{ \pm e_i \pm e_j \mid  \,i,j \in I_{\{ \pm 1 \} }(t) \} \cup  \{ \pm e_i \pm e_j \mid i, j \in I_{\{ \pm \omega_4 \} } (t) \} \cup \{ e_i - e_j \mid t_i=t_j \} \setminus \{ 0 \}.$$
We write $W(t)$ for the Weyl group of $\mathrm{C}_{\G}(t)$ relative to the torus $\T_0$.

\begin{lemma}
We have $|I_{ \{ \pm 1 \} }(t)|=|I_{  \{ \pm \omega_4 \} }(t) |=1$. 
\end{lemma}

\begin{proof}
Recall that $w_2 \in W$ satisfies ${}^{w_2} t=t z_2$ with $z_2=\prod_{i=1}^n \h_{e_i}(\omega_4)$. Therefore, we have  
$$I_{ \{ \pm 1 \} }({}^{w_2} t)=I_{  \{ \pm 1 \} }(t z_2) = I_{  \{ \pm \omega_4 \}}(t).$$
Thus, $w_2$ swaps the sets $I_{ \{ \pm 1 \} }(t)$ and $I_{  \{ \pm \omega_4 \} }(t)$. Hence, $|I_{ \{ \pm 1 \} }(t)|=|I_{  \{ \pm \omega_4 \} }(t) |$. Note that $I_{ \{ \pm 1 \}}(t)=I_{ \{ \pm 1 \} }({}^{w_1} t)$ and $I_{ \{ \pm \omega_4 \} }(t)=I_{ \{ \pm \omega_4 \} }({}^{w_1} t)$.

Suppose that $|I_{ \{ \pm 1 \} }(t)| > 1$ and let $a,b \in I_{ \{ \pm 1 \}}(t)$ with $a \neq b$. Fix $c,d \in I_{ \{ \pm \omega_4 \} }(t)$ with $c \neq d$ and let $w_1':=(a,-a) (d,-d) \in W$. It follows that ${}^{w_1'} t = t z_1$.

Recall that $e_a +e_b, e_a -e_b \in \Phi(t)$ and 
$$\mathrm{Z}(\C_\G(t))=\displaystyle\bigcap_{\alpha \in \Phi(t)} \mathrm{Ker}(\alpha).$$
Thus, for $\lambda=\prod_{i=1}^n \h_{e_i}(\lambda_i) \in \mathrm{Z}(\mathrm{C}_{\G}(t))$ we have $(\lambda_a \lambda_b^{ \pm 1})^2=1$. This implies that $\lambda_a$ and $\lambda_b$ are $4$th roots of unity. An analogue argument shows that $\lambda_c$ and $\lambda_d$ are also $4$th roots of unity. We conclude that ${}^{w_1'} \lambda=\lambda z_1$ or ${}^{w_1'} \lambda=\lambda$ in this case. 

Note that $ \pi ( \mathrm{C}_{\G}(t) )=\mathrm{C}^\circ_{\G^\ast}(s)$ by \cite[(2.2)]{Bonnafe}. From this we can conclude that
$$\pi ( \mathrm{Z}^\circ(\mathrm{C}_{\G}(t)) )= \mathrm{Z}^\circ(\mathrm{C}^\circ_{\G^\ast}(s)).$$
As ${}^{w_1'} \pi(\lambda)=\pi(\lambda)$ for all $ \lambda \in \mathrm{Z}(\mathrm{C}_\G(t))$ we conclude that $w_1' \in \Levi^\ast=\C_{\G^\ast}(\mathrm{Z}^\circ(\C^\circ_{\G^\ast}(s)))$. Since ${}^{w_1^{-1} w_1'} t = t$ it follows that $w_1^{-1} w_1' \in W(t)$. From this we deduce that $w_1 \in \Levi^\ast=\C_{\G^\ast}(\mathrm{Z}^\circ(\C^\circ_{\G^\ast}(s)))$. This contradicts the assumption $\mathbf{N}^\ast/ \Levi^\ast \cong \mathrm{Z}(\G)$.

We conclude that $|I_{ \{ \pm 1 \}}(t)| \leq 1$. By Remark \ref{empty} we must have $|I_{ \{ \pm 1 \} }(t)|=1$.
\end{proof}

By the previous lemma, up to a change of coordinates, we may assume that $I_{ \{ \pm 1 \} }(t)=\{ 1 \}$ and $I_{  \{ \pm \omega_4 \} }(t)=\{ n \}$.

\section*{Computations in the Weyl group}

Let us collect the information we have obtained so far. The root system $\Phi(t)$ of $\mathrm{C}_{\G}(t)$ is given by 
$$\Phi(t)=  \{ e_i - e_j \mid  t_i=t_j \} \setminus \{0 \}.$$
Observe that $\mathrm{C}_{\G}(t)$ is an $F$-stable Levi subgroup of $\G$ in duality with $\Levi^\ast$ so that $\Levi=\mathrm{C}_\G(t)$.

\begin{definition}
Let $I=\{2,\dots, n-1 \}$ and define 
$\overline{\Phi}':= \{ \pm e_i \pm e_j \mid i,j \in I \, \} \setminus \{ 0 \}$. Let 
$$\T_1:= \langle \h_{\alpha}(r) \mid r \in \overline{\mathbb{F}_q}^\times, \alpha \in \{e_1 \pm e_n \} \rangle$$
and $$\G_2:= \langle \h_{\bar{\alpha}}(\bar{r}),\x_{\alpha}(r) \mid \bar{\alpha} \in \overline{\Phi}', \alpha \in \Phi(t), r \in \overline{\mathbb{F}_q}, \bar{r} \in \overline{\mathbb{F}_q}^\times \rangle .$$
\end{definition}

The roots $\{ e_1 \pm e_n \}$ are orthogonal to those in $\bar{\Phi}'$ and no non-trivial linear combination of $\{e_1 \pm e_n \}$ and $\Phi'$ is a root in $\Phi$. Therefore, we have $\T_1  \subseteq	 \mathrm{Z}(\Levi)$. For $\T_2:=\G_2 \cap \T_0$ we have $\T_0=\T_1  \T_2$. This implies that $\Levi= \T_1 \G_2$. In addition, we have $\T_1 \cap \T_2 = \langle z_1 \rangle$.

\begin{lemma}\label{projection}
Consider the restriction map $$\mathrm{Res}: \{ \sigma  \in S_{ \{ \pm 1, \dots, \pm n \}} \mid \sigma(1),\sigma(n) \in \{ \pm 1, \pm n \} \} \to S_{ \{ \pm 1 , \pm n \} }.$$
We have $\mathrm{Res}(v) \in \langle (1,-1)(n,-n),  (1,-n)(-1,n) \rangle$.
\end{lemma}

\begin{proof}
Firstly, note that ${}^v F_0(s)=s$ which implies that $I_{ \{ \pm \omega_4,\pm 1 \} }({}^v s)=I_{ \{ \pm \omega_4,\pm 1 \} }(s)$. Therefore, $\mathrm{Res}(v)$ is well-defined.

Since $w_2$ permutes the sets $I_{ \{ \pm 1 \} }(t)=\{1 \}$ and $I_{  \{ \pm \omega \} }(t)=\{ n \}$ we have $\mathrm{Res}(w_2)=(1,-n)(-1,n)$. Let $w_1'=(1,-1)(n,-n) \in W$. Then we have ${}^{w_1'} t= t z_1$. This implies that $w_1' w_1^{-1} \in W(t)$. Since $W(t) \subseteq \mathrm{Ker}(\mathrm{Res})$ we must have $\mathrm{Res}(w_1)= (1,-1)(n,-n)$.

As $w_1,w_2 \in \mathrm{C}_W(v)$ we have $[\mathrm{Res}(w_i),\mathrm{Res}(v)]=1$ for $i=1,2$. Thus, 
$$\mathrm{Res}(v) \in \mathrm{C}_{ S_{ \{ \pm 1, \pm n \} } }( \langle (1,-1)(n,-n),(1,-n)(-1,n) \rangle ).$$
A short calculation shows that $\langle (1,-1)(n,-n),(1,-n)(-1,n) \rangle$ is self-centralizing in $S_{ \{ \pm 1 , \pm n \} }$.
\end{proof}

For the following two lemmas recall that $\x_{\alpha}(t)$, $\h_{\alpha}(t)$ and $\n_{\alpha}(t)$ are not uniquely defined and their relations depend on the choice of certain structure constants. However, the relations simplify in the case where the involved roots are orthogonal, see \cite[Remark 2.1.7]{BS2006}.

\begin{lemma}\label{orthogonal}
Let $A,B \subseteq \overline{\Phi}$ such that $A \perp B$. Let $x=\prod_{\alpha \in A} \n_\alpha(r_\alpha)$ and $y=\prod_{\beta \in B} \n_\beta(r_\beta)$ for $r_\alpha,r_\beta \in \overline{\mathbb{F}_q}^\times$. If $x,y \in \G$ then $x$ and $y$ commute.
\end{lemma}

\begin{proof}
Recall that the inclusion map $\mathrm{N}_\G(\T_0) \hookrightarrow \mathrm{N}_{\overline{\G}}(\T_0)$ induces the embedding $W \hookrightarrow \overline{W}$ such that $W= \mathrm{Ker}(\varepsilon)$. We note that $\varepsilon( \n_{\alpha}(1) \T_0)=-1$ for $\alpha \in \bar{\Phi}$ if and only if $\alpha$ is a short root. As $x,y \in \mathrm{N}_\G(\T_0)$ we deduce that the number of short roots in $A$ resp. $B$ is even.

Let $\alpha \in A$ and $\beta \in B$. By \cite[Remark 2.1.7(c)]{BS2006} we have
$\n_{\alpha}(r_\alpha)^{\n_{\beta}(r_\beta)}=\n_{\alpha}(-r_\alpha)=\h_{\alpha}(-1) \n_{\alpha}(r_\alpha)$, if either $\alpha$ or $\beta$ is a long root. On the other hand, we have
$\n_{\alpha}(r_\alpha)^{\n_{\beta}(r_\beta)}=\n_{\alpha}(r_\alpha)$ if both $ \alpha$ and $\beta$ are short roots. Note that if $\alpha$ is a short root then $\h_\alpha(-1)=\h_{e_1}(-1)$. The result follows from this.
\end{proof}

In the following, we will consider the element
$$\n_1:=\n_{e_1}(1) \n_{e_1}(1)^{\n_{e_1-e_n}(1)} \in \mathrm{N}_{\G}(\T_0),$$
which is a preimage of $w_1'=(1,-1)(n,-n) \in W$. By the proof of Lemma \ref{projection} it is possible to find $\n_2' \in \langle n_{e_i}(1) \mid i \in I  \rangle \cap \mathrm{N}_\G(\T_0)$ such that the element
$$\n_2:= \n_{e_1-e_n}(1) \n_2' \in \mathrm{N}_{\G}(\T_0)$$
is a preimage of $w_2 \in W$.

\begin{lemma}\label{commute}
The elements $\n_1$ and $\n_2$ commute. In addition, we have $\n_1 \in \mathrm{C}_{\G}(\G_2)$.
\end{lemma}

\begin{proof}
Let us first prove that $\n_1$ and $\n_2$ commute. By \cite[Remark 2.1.7(c)]{BS2006} we have
$\n_{e_j}(u)^{\n_{e_i}(1)}=\n_{e_j}(-u)=\h_{e_j}(-1) \n_{e_j}(u)$ for $u \in \overline{\mathbb{F}_q}$, whenever $i\neq j$. By the relation in \cite[Theorem 2.1.6(b)]{BS2006} we have $\n_{e_1}(1)^{\n_{e_1-e_n}(1)}\in \{\n_{e_n}(1),\n_{e_n}(-1)\}$. By Lemma \ref{orthogonal},
$$\n_1^{n_2}= \n_1^{\n_{e_1-e_n}(1)}= \n_{e_1}(1)^{\n_{e_1-e_n}(1)}\n_{e_1}(1)^{\n_{e_1-e_n}(1)^2}= \n_{e_1}(1)^{\n_{e_1-e_n}(1)}\n_{e_1}(1)^{\h_{e_1-e_n}(-1)}.$$
According to \cite[Remark 2.1.8]{BS2006} we have $\h_{e_1-e_n}(-1)=\h_{e_1}(\omega_4) \h_{e_n}(\omega_4^{-1})$ , where $\omega_4 \in \overline{\mathbb{F}_q}^\times$ is a fourth root of unity. Using \cite[Remark 2.1.7(a)]{BS2006}, we obtain
$$ \n_{e_1}(1)^{\n_{e_1-e_n}(1)}\n_{e_1}(1)^{\h_{e_1-e_n}(-1)}=\n_{e_1}(1)^{\n_{e_1-e_n}(1)}\n_{e_1}(1)^{\h_{e_1}(\omega_4)}=\n_{e_1}(1)^{\n_{e_1-e_n}(1)}\n_{e_1}(1) z_1.$$
Since $\n_{e_1}(1)^{\n_{e_1-e_n}(1)}\in \{\n_{e_n}(1),\n_{e_n}(-1)\}$ we deduce that 
$$\n_{e_1}(1)^{\n_{e_1-e_n}(1)}\n_{e_1}(1) z_1=\n_{e_1}(1)\n_{e_1}(1)^{\n_{e_1-e_n}(1)}=\\n_1.$$
Therefore, $\n_1^{n_2}=\n_1$ and we conclude that $\n_1$ and $\n_2$ commute.

% Since $\m_2 \in  \langle n_{e_i}(r) \mid i \in I, r \in \overline{\mathbb{F}_q}^\times  \rangle \cap \mathrm{N}_\G(\T_0)$ it follows again by Lemma \ref{orthogonal} that $m_2$ and $\n_{e_1-e_n}(1)$ commute. Since $F(\n_2')=\n_2'$ it follows that $F(\n_2)=\n_2$. We conclude that $F(\n_2)=\n_2$.

Finally, note that $\n_1 \in \mathrm{C}_\G(\G_2)$ by \cite[Remark 2.1.7(b)]{BS2006} and \cite[Theorem 2.1.6(c)]{BS2006}.
\end{proof}

By Lemma \ref{projection} there exist $\m_1  \in \langle \n_1, \n_{e_1-e_n}(1) \rangle$ and $\m_2 \in  \langle \n_{e_i}(1) \mid i \in I  \rangle \cap \mathrm{N}_\G(\T_0)$ such that 
$$\m:=\m_1 \m_2$$
satisfies $\m \T_0 = v$ in $W$. Since $w_2 \in \C_W(v)$ we necessarily have $(\m F_0)(\n_2) \n_2^{-1} \in \T_0$. By Lemma \ref{orthogonal} it follows that $\m_2$ commutes with $\n_{e_1-e_n}(1)$ and $\m_1$ commutes with $\n_2'$. From this we deduce that
$$(\m F)(\n_2) \n_2^{-1}={}^{\m} \n_2 \n_2^{-1}={}^{\m_2} \n_2' \n_2'^{-1}.$$
Since ${}^{\m_2} \n_2' \n_2'^{-1}$ is purely an expression in the roots $e_2,\dots,e_{n-1}$ we can deduce that 
$$(\m F)(\n_2) \n_2^{-1} \in \T_2 = \langle \h_{e_i}(r) \mid i \in I, r\in \overline{\mathbb{F}_q}^\times \rangle.$$
By Lang's theorem there exists $t_2 \in \T_2$ such that $(t_2 \m F_0)(\n_2)=\n_2$. Replacing $\m_2$ by $t_2 \m_2 \in \langle \n_{e_i}(r_i) \mid, r_i \in \overline{\mathbb{F}_q}^\times, i \in I  \rangle \cap \mathrm{N}_\G(\T_0)$ we can henceforth assume that $(\m F_0)(\n_2)=\n_2$.

By Lemma \ref{commute} it follows that $\m_1\in \langle \n_1, \n_{e_1-e_n}(1) \rangle$ commutes with $\n_1$. By Lemma \ref{orthogonal} we conclude that $\m_2$ and $\n_1$ commute. As $\n_1$ is $F_0$-stable it follows that $(\m F_0)(\n_1)=\n_1$.

If $y \in \T_0$ then we have an isomorphism $\G^F \to \G^{F y}, \, g \mapsto {}^{y^{-1}} g$ which yields isomorphic fixed-point structures for all relevant subgroups. We may thus fix a nice representative of $ v\in W$ in $\mathrm{N}_{\G}(\T_0)$ and assume the following:

\begin{assumption}
From now on, we suppose that $F=\m F_0$. In particular, the elements $\n_1,\n_2$ are $F$-stable.
\end{assumption}

 We are now ready to prove the main result of this section.

\begin{proposition}
We have $\Levi^F \langle \n_1, n_2 \rangle= \mathbf{N}^F$.
\end{proposition}
\begin{proof}
The elements $\n_1,\n_2 \in \mathrm{N}_\G(\T_0)$ satisfy ${}^{\n_1} t =t z_1$
and ${}^{\n_2} t=t z_2$. From this we deduce that $\pi(\n_1), \pi(\n_2) \in \C_{\G^\ast}(s)$. By duality we have an isomorphism $\mathbf{N}^F/ \Levi^F \cong (\mathbf{N}^\ast)^{F^\ast} / (\Levi^\ast)^{F^\ast}$ from which we can now conclude that $\Levi^F \langle \, \n_2 \rangle= \mathbf{N}^F$.
\end{proof}

In the next section we will consider the subgroup $L_0$ of $\Levi^F$ defined by $L_0=\T_1^F \G_2^F$. As $\T_1  \subseteq	 \C_\Levi(\G_2)$ it follows that $L_0$ is a central product of $\T_1^F$ and $\G_2^F$. The following lemma shows that $\Levi^F/L_0 \cong C_2$.

\begin{lemma}\label{group}
Let $\mathcal{L}: \G \to \G, \, g\mapsto g^{-1} F(g)$, denote the Lang map of $\G$. There exists $x_1\in \T_1$ and $x_2 \in \T_2$ such that $\mathcal{L}(x_1)=\mathcal{L}(x_2)=h_{e_1}(-1)$ and $x:=x_1 x_2$ satisfies $\Levi^F=\T_1^F \G_2^F \langle x \rangle$.
% and $t^{x_1}=t^{-1}$ for all $t\in \T_1^F$.  
\end{lemma}

\begin{proof}
The existence of $x_1$ and $x_2$ follows by applying Lang's theorem. Since $\T_1 \cap \G_2= \T_1 \cap \T_2= \langle h_{e_1}(-1) \rangle$ the second claim follows.
\end{proof}

\section*{Representation theory}

Let $\hat{s}: \mathcal{O} \Levi^F \to \mathcal{O}^\times$ be the character of $\Levi^F$ corresponding to the central element $s \in \mathrm{Z}(\Levi^\ast)$, see \cite[Equation 8.19]{MarcBook}.

\begin{lemma}\label{linear} The linear character $\hat{s}: \Levi^F \to \mathcal{O}^\times$ extends to $\N^F$.
\end{lemma}

\begin{proof}
By \cite[Theorem 1.1]{BS} the character $\lambda:=\mathrm{Res}^{\Levi^F}_{\T_0^F}(\hat{s})$ extends to its inertia group in $\N_{\G^F}(\T_0^F)$. However, $\n_1, \n_2 \in \N_{\G^F}(\T_0^F)$ and $\lambda$ is $\mathbf{N}^F$-invariant which implies that $\lambda$ extends to a character $\lambda'$ of $\T_0^F \langle \n_1, \n_2 \rangle$. We define a character $\hat{s}': \mathbf{N}^F \to \mathcal{O}^\times$ by $\hat{s}'(x):= \hat{s}(l) \lambda'(n)$ where $x \in \mathbf{N}^F$ with $x=ln$ for $l \in \Levi^F$ and $n \in \T_0^F \langle \n_1, \n_2 \rangle$. Note that this character is well-defined as $\hat{s}$ and $\lambda$ agree on the intersection $\Levi^F \cap \T_0^F \langle \n_1,\n_2 \rangle= \T_0^F $. 
\end{proof}

The following lemma is a module theoretic generalization of \cite[Lemma 4.1]{BS}.

\begin{lemma}\label{module}
Let $\widetilde{Y}$ be a finite group with normal subgroup $\widetilde{X}$ and subgroup $Y$ such that $\widetilde{Y}=Y \widetilde{X}$. Denote $X:=Y \cap \widetilde{X}$ and suppose that $\ell \nmid  [Y:X]$. Suppose that $M$ is an indecomposable $\mathcal{O}X$-module which extends to an $\mathcal{O} Y$-module and suppose that $\widetilde{M}$ is an $\mathcal{O}\widetilde{X}$-module such that $M=\mathrm{Res}^{\widetilde{X}}_X(\widetilde{M})$. If $\widetilde{M}$ is $\widetilde{Y}$-invariant then $\widetilde{M}$ extends to $\widetilde{Y}$.
\end{lemma}

\begin{proof}
Let us recall some basic facts about Clifford theory, see \cite[Section 7.B]{Dat} (over $k$) and \cite{Dade} (over $\mathcal{O}$). We follow the notation in  \cite[Section 7.B]{Dat}.

Firstly, for $y \in Y$, define
$$N_y:=\{ \phi \in \End^\times_{\mathcal{O}}(M) \mid \phi(xm)=yxy^{-1} \phi(m) \text{ for all } x \in X, m \in M \}$$
and let $N:=\cup_{y\in Y} N_y$. Note that $N$ is a group with normal subgroup $N_1$. Since $M$ is $Y$-invariant we have a surjective morphism $Y \to N/N_1$ given by $y \mapsto y N_1$. We form the group
$$\widehat{Y}:=Y \times_{N/N_1} N= \{ (y, \varphi) \mid \varphi \in N_y \}.$$
We let $A:= \End_{\mathcal{O} X}(M)$.
Consider the following exact sequence:
$$1 \to A^\times \to \widehat{Y} \to Y \to 1.$$
The $\mathcal{O} X$-module $M$ extends to an $\mathcal{O} Y$-module if and only if this sequence splits, see \cite[1.7]{Dade}

The action of $X$ on $M$ defines $\phi_x \in N_x$ for every $x \in X$. We identify $X$ with its image under the diagonal embedding $X \hookrightarrow \widehat{Y}, x \mapsto (x, \phi_x)$. As $\ell \nmid [Y:X]$ it follows (see \cite[Theorem 4.5]{Dade}) that $M$ extends to a $\mathcal{O} Y$-module if and only if the following exact sequence splits:
$$1 \to A^\times /(1+J(A)) \to \widehat{Y}/X (1+J(A)) \to Y/X \to 1.$$

Similarly, we can look at $\widetilde{M}$ instead of $M$. We denote the corresponding objects with a tilde. Analogously, the module $\widetilde{M}$ extends to an $\mathcal{O} \widetilde{Y}$ module if and only if the following exact sequence splits:
$$1 \to \widetilde{A}^\times/(1+J(\widetilde{A})) \to \widehat{\widetilde{Y}}/\widetilde{X} (1+J(\widetilde{A})) \to \widetilde{Y}/\widetilde{X} \to 1.$$ 
Let $\pi: \widetilde{Y}/\widetilde{X} \to Y/X$ be the inverse map of the natural isomorphism $Y/X  \to \widetilde{Y}/\widetilde{X}$. Restriction defines a homomorphism $\widetilde{A}^\times \to A^\times$.

Now we define a map $\widehat{\widetilde{Y}} \to \widehat{Y}$ as follows. For $(\widetilde{y},\phi) \in \widehat{\widetilde{Y}}$ we let $\tilde{x} \in \widetilde{X}$ such that $y:=\tilde{y} \tilde{x} \in Y$. Let $\phi_{\tilde{x}}$ be the natural action of $\widetilde{x}$ on $\widetilde{M}$. Then it follows that $\phi \phi_{\tilde{x}} \in \widetilde{N}_{y} \subseteq N_y$. We define 
$$\widehat{\pi}: \widehat{\widetilde{Y}}/\widetilde{X} \to \widehat{Y}/X, \, (\widetilde{y},\phi) \mapsto (y,\phi \phi_{\widetilde{x}}).$$
Note that if $\tilde{x}' \in \widetilde{X}$ with $y':=\tilde{y} \tilde{x}' \in Y$ then $x:=\tilde{x}^{-1} \tilde{x}' \in X$ and we have $y'=y x$. From this we deduce that $(y',\phi \phi_{\tilde{x}'})=(y, \phi \phi_{x})$ in $\tilde{Y}/X$ which shows that the map $\widehat{\pi}$ is well-defined. We can therefore consider the following commutative diagram:
\begin{center}
\begin{tikzpicture}
  \matrix (m) [matrix of math nodes,row sep=3em,column sep=3em,minimum width=2em] {
 1 & \widetilde{A}^\times /(1+J(\widetilde{A})) &  \widehat{\widetilde{Y}}/\widetilde{X} (1+J(\widetilde{A}))  & \widetilde{Y}/\widetilde{X} & 1 \\  
    1 & A^\times /(1+J(A)) &  \widehat{Y}/X (1+J(A)) & Y/X & 1  \\};
\path[-stealth]
(m-1-2) edge node [left] {} (m-2-2)
(m-1-3) edge node [left] {$\widehat{\pi}$} (m-2-3)
(m-1-4) edge node [left] {$\pi$} (m-2-4)

(m-1-1) edge node [left] {} (m-1-2)
(m-1-2) edge node [above] {} (m-1-3)
(m-1-3) edge node [above] {} (m-1-4)
(m-1-4) edge node [above] {} (m-1-5)

(m-2-1) edge node [left] {} (m-2-2)
(m-2-2) edge node [above] {} (m-2-3)
(m-2-3) edge node [above] {} (m-2-4)
(m-2-4) edge node [above] {} (m-2-5);
\end{tikzpicture}
\end{center}
Now note that $\pi$ is an isomorphism. Also as $M$ and $\widetilde{M}$ are indecomposable we have that $A^\times/(1+J(A)) \cong k^\times$ resp. $\widetilde{A}^\times/(1+J(\widetilde{A}))\cong k^\times$. Thus, the first and the third vertical map are isomorphisms. By the five lemma, it follows that $\widehat{\pi}: \widehat{\widetilde{Y}}/\widetilde{X}(1+J(\widetilde{A})) \to \widehat{Y}/X (1+J(A))$ is also an isomorphism. Thus, the two extensions are isomorphic. However, by assumption we already know that $M$ extends to an $\mathcal{O} Y$-module which implies that the sequence in the second row splits. Thus, also the sequence of the first row splits and $\widetilde{M}$ extends to an $\mathcal{O}\widetilde{Y}$-module.
\end{proof}

We are now ready to prove the main statement of this section.

\begin{proposition}\label{main}
Let $M$ be a $\mathbf{N}^F$-invariant indecomposable $\mathcal{O} \G^F$-$\mathcal{O} \Levi^F e_s^{\Levi^F}$-bimodule. If $\ell$ does not divide $q^2-1$ then $M$ extends to an $\mathcal{O} \G^F$-$\mathcal{O} \mathbf{N}^F$-bimodule.
\end{proposition}

\begin{proof}
By Lemma \ref{linear}, it follows that $M$ extends to $ \G^F \times (\mathbf{N}^F)^{\mathrm{opp}}$ if and only if $M \otimes_{\mathcal{O}} \hat{s}^{-1}$ extends to $\mathcal{O}\G^F \times (\mathbf{N}^F)^{\mathrm{opp}}$. We may therefore assume from now on that $M$ is an indecomposable $\mathcal{O} \G^F$-$\mathcal{O} \mathbf{L}^F e_1^{\Levi^F}$-bimodule.

Since $\ell \nmid [\Levi^F: L_0]$ there exists an indecomposable $\mathcal{O} \G^F$-$\mathcal{O} L_0$-bimodule $M_0$ such that $M$ is a direct summand of $\Ind^{\G^F \times (\Levi^F)^{\mathrm{opp}}}_{\G^F \times L_0^{\mathrm{opp}}}(M_0)$. As $1 \times (\T^F_1)^{\mathrm{opp}}$ is central in $\G^F \times {L_0}^{\mathrm{opp}}$ we deduce that
$$\mathrm{Res}^{\G^F \times {L_0}^{\mathrm{opp}}}_{1 \times {(\T_1^F)_{\ell'}}^{\mathrm{opp}}}(M_0)=S^{\mathrm{dim}(M_0)}$$
for some simple $\mathcal{O} (\T_1)_{\ell'}^F$-module $S$. Let $\lambda:  (\T_1)_{\ell'}^F \to \mathcal{O}^\times$ be the character corresponding to $S$. Since $\mathrm{Res}^{\G^F \times {\Levi^F}^{\mathrm{opp}}}_{1 \times {\Levi^F}}$ is a $\mathcal{O}\Levi^F e_1^{\Levi^F}$-module it follows that $\lambda$ is a character in a unipotent block, which implies that $\lambda$ is the trivial character.

Note that $|\T_1^F|\in \{(q-1)^2,(q+1)^2 \}$ and therefore $\ell \nmid |\T_1^F|$ by assumption. We conclude that 
$$\mathrm{Res}^{\G^F \times {L_0}^{\mathrm{opp}}}_{1 \times {(\T_1^F)}^\mathrm{opp}}(M_0)=\mathcal{O}^{\mathrm{dim}(M_0)},$$
where $\mathcal{O}$ is the trivial $\mathcal{O}{(\T_1^F)}^{\mathrm{opp}}$-module.
Since $L_0/\T_1^F \cong \G^F_2/ \langle z_1 \rangle$ we may consider $M_0$ as an indecomposable $\mathcal{O}[\G^F \times (\G^F_2/\langle z_1 \rangle)^{\mathrm{opp}}]$-module.

The element $\n_1$ centralizes $\G^F_2$ and hence we can extend $M_0$ to an $\mathcal{O}\G^F \times (L_0 \langle \n_1 \rangle)^{\mathrm{opp}}$-module by letting $\n_1$ act trivially on $M_0$. We denote this extension by $M_1$.

Since $M$ is a direct summand of $\Ind^{\G^F \times (\Levi^F)^{\mathrm{opp}}}_{\G^F \times L_0^{\mathrm{opp}}}(M_0)$ it follows that $\Res^{\G^F \times (\Levi^F)^{\mathrm{opp}}}_{\G^F \times(L_0)^{\mathrm{opp}}} M$ is a direct summand of 
$$\Res^{\G^F \times (\Levi^F)^{\mathrm{opp}}}_{\G^F \times(L_0)^{\mathrm{opp}}} \Ind^{\G^F \times (\Levi^F)^{\mathrm{opp}}}_{\G^F \times L_0^{\mathrm{opp}}}(M_0) \cong M_0 \oplus  M_0^x,$$
where $x=x_1 x_2 \in \Levi^F$ as in Lemma \ref{group}. As the quotient group $\Levi^F/L_0$ is cyclic of $\ell'$ order it follows by \cite[Lemma 10.2.13]{Rouquier3} that either $\Res^{\G^F \times (\Levi^F)^{\mathrm{opp}}}_{\G^F \times(L_0)^{\mathrm{opp}}}(M)\cong M_0 $ or $\Res^{\G^F \times (\Levi^F)^{\mathrm{opp}}}_{\G^F \times (L_0)^{\mathrm{opp}}}(M)\cong M_0 \oplus  M_0^x$. We treat these two cases separately.
\\
\\
\textbf{Case 1:} Assume that $\Res^{\G^F \times (\Levi^F)^{\mathrm{opp}}}_{\G^F \times(L_0)^{\mathrm{opp}}}(M)\cong M_0 $. \\
Since $M$ is $\mathbf{N}^F$-invariant it follows that $M_0$ is $\mathbf{N}^F$-invariant.

We have $[\n_1,\n_2]=1$. Thus, the action of $\n_1$ on $ M_1^{\n_2}$ is equal to the action of $\n_2 \n_1 \n_2^{-1}=\n_1$ on $M_1$. However, $\n_1$ acts trivially on $M_0$. Since $M_0$ is $\n_2$-invariant there exists an isomorphism $\phi: M_0 \to M_0^{\n_2}$ of $\G^F \times(L_0)^{\mathrm{opp}}$-modules. Since $\n_1$ acts trivially on $M_1$ it follows that $\phi: M_1 \to M_1^{\n_2}$ is an isomorphism of $\mathcal{O}\G^F \times (L_0 \langle \n_1 \rangle)^{\mathrm{opp}}$-modules or in other words $M_1$ is $\n_2$-invariant. From this we conclude that $M_1$ extends to $\G^F \times (L_0 \langle \n_1, \n_2 \rangle)^{\mathrm{opp}}$.

Applying Lemma \ref{module} to $\tilde{X}= \G^F \times (\Levi^F)^{\mathrm{opp}}$ and $Y= \G^F \times (L_0 \langle \n_1,\n_2 \rangle)^{\mathrm{opp}}$) implies that $M$ extends to a $\G^F \times (\mathbf{N}^F)^{\mathrm{opp}}$-module.
\\
\\
\textbf{Case 2:} Assume that $\Res^{\G^F \times (\Levi^F)^{\mathrm{opp}}}_{\G^F \times (L_0)^{\mathrm{opp}}}(M)\cong M_0 \oplus {}^x M_0$.
%for $x=x_1 x_2 \in \Levi^F$ as in Lemma \ref{group}.
\\
We note that $ M_0^{\n_1} \cong M_0$. On the other hand, we either have $ M_0^{\n_2} \cong M_0$ or $ M_0^{\n_2 x} \cong M_0$.

Suppose that $ M_0^{\n_2} \cong M_0$. Then $M_0$ is $\mathbf{N}^F$-invariant. Using the same proof as in case 1 we deduce that $M_0$ extends to a $\G^F \times (L_0 \langle \n_1,\n_2 \rangle)^{\mathrm{opp}}$-module.

Suppose that $ M_0^{\n_2 x} \cong M_0$. We have $h_{e_1}(-1)^{x}=h_{e_1}(-1)$ as $\mathcal{L}(x_1)=\mathcal{L}(x_2)=h_{e_1}(-1)$. Since $x_2 \in \G_2$ we conclude that $\n_1^{x_2}=\n_1$. Therefore,
$\n_1^{\n_2 x}=\n_1^x =\n_1^{x_1}.$

Clearly, $x_1 x_1^{\n_1^{-1}} \in \T_1^F$ which implies that $\n_1^{x_1} \n_1^{-1} \in \T_1^F$. From this we deduce that $\n_1^{\n_2 x} \n_1^{-1} \in \T_1^F$.
Now $\n_1$ acts on $M_1^{\n_2 x}$ as $\n_1^{\n_2 x}$ acts on $M_1$. Since $\T^F_1$ and $\n_1$ act trivially on $M_1$ it follows that $\n_1$ acts trivially on $M_1^{\n_2 x}$. Since $M_0$ is $\n_2 x$-invariant it follows that $M_1$ is $\n_2 x$-invariant. Thus, $M_0$ extends to a $\G^F \times (L_0 \langle \n_1,\n_2 x \rangle)^{\mathrm{opp}}$-module.

It follows that $M_0$ extends to a $\G^F \times (L_0 \langle \n_1, \n_2' \rangle)^{\mathrm{opp}}$-module $M'$, where $\n_2' \in \{ \n_2, \n_2 x \}$. By Mackey's formula we deduce that 
$$\Res_{\G^F \times (\Levi^F)^{\mathrm{opp}}}^{\G^F \times (\mathbf{N}^F)^{\mathrm{opp}}}\mathrm{Ind}_{\G^F \times (L_0 \langle \n_1, \n_2' \rangle)^{\mathrm{opp}}}^{\G^F \times (\mathbf{N}^F)^{\mathrm{opp}}} M' \cong \mathrm{Ind}_{\G^F \times L_0^{\mathrm{opp}}}^{\G^F \times (\Levi^F)^{\mathrm{opp}}} \mathrm{Res}_{\G^F \times (L_0)^{\mathrm{opp}}}^{\G^F \times (L_0 \langle \n_1,\n_2' \rangle)^{\mathrm{opp}}} M'\cong M.$$
Thus, $\mathrm{Ind}_{\G^F \times (L_0 \langle \n_1, \n_2' \rangle)^{\mathrm{opp}}}^{\G^F \times (\mathbf{N}^F)^{\mathrm{opp}}} M'$ is an extension of $M$ to $\G^F \times (\mathbf{N}^F)^{\mathrm{opp}}$. This finishes the proof.
\end{proof}

Using a standard argument in Clifford theory we can now deduce Proposition \ref{ext} from the previous proposition.

\begin{proof}[Proof of Proposition \ref{ext}] According to \cite[Theorem 7.2]{Dat} the bimodule $H^d(\Y_\U^\G,\mathcal{O}) e_s^{\Levi^F}$ is $\mathcal{O} \mathbf{N}^F$-invariant. Let $H^d(\Y_\U^\G,\mathcal{O}) e_s^{\Levi^F}= \oplus_{i=1}^{k} M_i$ be a decomposition into $\mathbf{N}^F$-orbits of indecomposable direct summands of $H^d(\Y_\U^\G,\mathcal{O}) e_s^{\Levi^F}$. Let $N_i$ be an indecomposable direct summand of $M_i$ and $T_i$ be its inertia group in $\mathbf{N}^F$. If $T_i$ is a proper subgroup of $\mathbf{N}^F$ then $T_i/ \Levi^F$ is cyclic of $\ell'$-order so that $N_i$ extends to $\G^F \times (T_i)^{\mathrm{opp}}$. If $T_i= \mathbf{N}^F$ then $N_i$ extends to $\G^F\times (\mathbf{N}^F)^{\mathrm{opp}}$ by Proposition \ref{main}. Let $N_i'$ be an extension of $N_i$ to $\G^F \times (T_i)^{\mathrm{opp}}$. By Clifford theory, it follows that $\mathrm{Ind}^{\G^F \times (\mathbf{N}^F)^{\mathrm{opp}}}_{\G^F \times (T_i)^{\mathrm{opp}}} N_i'$ is an extension of $M_i$. This shows that $H^d(\Y_\U^\G,\mathcal{O}) e_s^{\Levi^F}$ extends to $\G^F \times (\mathbf{N}^F)^{\mathrm{opp}}$.
\end{proof}

\section*{Proof of Morita equivalence}

In this final section we prove that the extended bimodule induces a Morita equivalence. The following section borrows arguments from \cite{Note}.

From now on let $\G$ be a connected reductive group. We keep the notation as in \cite[Section 7.C]{Dat}. In particular we fix a regular embedding $ \G \hookrightarrow \Gtilde$. We let $\Ltilde= \Levi \mathrm{Z}(\Gtilde)$ and $\tilde{\mathbf{N}}=\mathbf{N} \Ltilde$.

\begin{proposition}\label{Morita}
Suppose that the $\mathcal{O}\G^F$-$\mathcal{O} \Levi^F$-bimodule $H_c^d(\Y_\U^\G,\mathcal{O}) e_s^{\Levi^F}$ extends to an $\mathcal{O}\G^F$-$\mathcal{O} \N^F$-bimodule $M'$. Then the bimodule $M'$ induces a Morita equivalence between $\mathcal{O} \N^F e_s^{\Levi^F}$ and $\mathcal{O} \G^F e_s^{\G^F}$.
\end{proposition}

\begin{proof}
Let  $M'$ be an $\mathcal{O} \G^F \times (\mathbf{N}^F)^{\mathrm{opp}}$-bimodule extending $M:=H^d_c(\Y_\U^\G, \mathcal{O}) e_s^{\Levi^F}$. Recall that $M$ is projective as $\mathcal{O} \G^F$-module and projective as $\mathcal{O} \Levi^F$-module. As $\ell \nmid [\mathbf{N}^F : \Levi^F]$ it follows that $M'$ is projective as $\mathcal{O} \mathbf{N}^F$-module. Note that $\Ind_{\G^F}^{\Gtilde^F} M$ is a faithful $\mathcal{O} \Gtilde^F e_s^{\G^F}$-module, see proof of \cite[Theorem 7.5]{Dat}. Thus, $M$ is a faithful $\mathcal{O} \G^F e_s^{\G^F}$-module.

By \cite[Theorem 0.2]{Broue} it suffices to prove that $M' \otimes_{\mathcal{O}} K$ induces a bijection between irreducible characters of $K \mathbf{N}^F e_s^{\Levi^F}$ and $K \G^F e_s^{\Levi^F}$. As $M$ is a faithful $\mathcal{O} \G^F e_s^{\G^F}$-module it suffices to prove that $\End_{K \G^F}(M' \otimes_\mathcal{O} K)\cong K \mathbf{N}^F e_s^{\Levi^F}$. As in the proof of \cite[Theorem 7.5]{Dat} let $\widetilde{M}= \Ind_{\G^F \times (\Levi^F)^{\mathrm{opp}} \Delta \Ltilde^F}^{\Gtilde^F \times (\Ltilde^F)^{\mathrm{opp}}}(M)$ be the $\mathcal{O} \Gtilde^F \times (\widetilde{\Levi}^F)^{\mathrm{opp}}$-module. We have $\Ind_{\G^F}^{\Gtilde^F} M' \cong \widetilde{M}$ as $\Gtilde^F$-modules . Since $M$ is $\Gtilde^F$-invariant this implies 
$$\mathrm{dim}(\End_{K \G^F}(M))= [\Gtilde^F: \G^F] \, \mathrm{dim}(\End_{K\Gtilde^F}(\widetilde{M})).$$
In addition, the bimodule $\widetilde{M}$ extends to an $\mathcal{O} \Gtilde^F \times (\widetilde{\mathbf{N}}^F)^{\mathrm{opp}}$-bimodule $\widetilde{M}'$, see proof of \cite[Theorem 7.5]{Dat}, which induces a Morita equivalence between $\mathcal{O} \widetilde{\G}^F e_s^{\G^F}$ and $\mathcal{O} \widetilde{\mathbf{N}}^F e_s^{\Levi^F}$. This shows that $\mathrm{dim}(\End_{K\Gtilde^F}(\widetilde{M})) = \mathrm{dim}(K \widetilde{\mathbf{N}}^F e_s^{\Levi^F})$. Moreover, we have 
$$\mathrm{dim}(K \widetilde{\mathbf{N}}^F e_s^{\Levi^F})= [\widetilde{\mathbf{N}}^F: \mathbf{N}^F] \, \mathrm{dim}(K \mathbf{N}^F e_s^{\Levi^F}).$$
From these calculations using $\Gtilde^F/ \G^F \cong \widetilde{\mathbf{N}}^F / \mathbf{N}^F$ we deduce that 
$$\mathrm{dim}(\End_{K \G^F}(M))= \mathrm{dim}(K \mathbf{N}^F e_s^{\Levi^F}).$$

\begin{lemma}
The natural map $K \mathbf{N}^F e_s^{\Levi^F} \to \End_{K \G^F}(M')$ is injective.
\end{lemma}
\begin{proof}
Let $\dot{n}$ be a representative of $n \in \mathbf{N}^F/\Levi^F$ in $\mathbf{N}^F$. 
Let $\alpha_n \in \mathcal{O} \Levi^F e_s^{\Levi^F}$, $n \in \mathbf{N}^F/\Levi^F$, such that $\sum_{n \in \mathbf{N}^F / \Levi^F} \alpha_n \dot{n}=0$ on $M'$. Let $\theta_{\dot{n}}$ be the automorphism on $\widetilde{M}$ induced by the action of $\dot{n}$. More concretely, we have 
$$\theta_{\dot{n}}((g,l) \otimes m)=(g,{}^n l) \otimes \dot{n} m$$
for $(g,l) \in \Gtilde^F \times (\Ltilde^F)^{\mathrm{opp}}$ and $m \in M$.

For $(g,l) \in \Gtilde^F \times (\Ltilde^F)^{\mathrm{opp}}$ we have $\theta_{\dot{n}} \circ (g,l) \circ \theta_{\dot{n}}^{-1}=(g,{}^n l)$ on $\widetilde{M}$. Let $e \in \mathrm{Z}(\mathcal{O} \Ltilde^F)$ be the central idempotent as in \cite[Theorem 7.5]{Dat} such that $e_s^{\Levi^F}= \sum_{n \in \mathbf{N}^F/ \Levi^F } {}^n e$. We have
$$\widetilde{M}=\displaystyle \bigoplus_{n\in \mathbf{N}^F/\Levi^F} \widetilde{M} {}^n e.$$
For $\widetilde{m} \in \widetilde{M}e$ we therefore have
$$\sum_{n\in  \mathbf{N}^F / \Levi^F} \alpha_n \theta_{\dot{n}}(\widetilde{m})=0.$$
As $\theta_{\dot{n}}(\widetilde{m}) \in \widetilde{M} {}^n e$ we have $\alpha_{n} \theta_{\dot{n}}(\widetilde{m})=0$ for all $n \in \mathbf{N}^F/\Levi^F$ and $\widetilde{m} \in \widetilde{M}e$. This means that $\alpha_{n} \theta_{\dot{n}}$ vanishes on $\widetilde{M} e$. Composing with $\theta_{\dot{y}}^{-1}$ for $y \in \mathbf{N}^F/ \Levi^F$ shows that $\alpha_{n} \theta_{\dot{n}}$ vanishes on $\widetilde{M} {}^y e$ as well. We conclude that $\alpha_n \theta_n=0$ on $\widetilde{M}$. As $\theta_{\dot{n}}$ is an isomorphism we must have $\alpha_n=0$ on $\widetilde{M}$ and therefore, $\alpha_n=0$ on $M$. As $M$ is faithful as $K \Levi^F e_s^{\Levi^F}$-module it follows that $K \mathbf{N}^F e_s^{\Levi^F} \to \End_{K \G^F}(M')$ is injective.
\end{proof}
Now let us finish the proof of Proposition \ref{Morita}. Since $\mathrm{dim}(\End_{K \G^F}(M))= \mathrm{dim}(K \mathbf{N}^F e_s^{\Levi^F})$ it follows that the natural map $$K \mathbf{N}^F e_s^{\Levi^F} \to \End_{K \G^F}(M')$$ is an isomorphism. Thus, the bimodule $M' \otimes_{\mathcal{O}} K$ induces a Morita equivalence between $K \mathbf{N}^F e_s^{\Levi^F}$ and $K \G^F e_s^{\Levi^F}$. As we have argued above this implies that $M'$ induces a Morita equivalence between $\mathcal{O} \N^F e_s^{\Levi^F}$ and $\mathcal{O} \G^F e_s^{\G^F}$.
\end{proof}

The previous statement is now sufficient to prove Theorem \ref{final}:

\begin{proof}[Proof of Theorem \ref{final}] The quotient $\mathbf{N}^F/ \Levi^F$ is of $\ell'$-order and embeds into $\mathrm{Z}(\G)^F$, see for instance \cite[Lemma 13.16(i)]{MarcBook}. Thus, the quotient $\mathbf{N}^F / \Levi^F$ is cyclic of $\ell'$-order unless possibly if $\G$ is simply connected and $\G^F$ is of untwisted type $D_{n}$, $n$ even. If $\mathbf{N}^F/ \Levi^F$ is cyclic (of $\ell'$-order) then it follows by \cite[Lemma 10.2.13]{Rouquier3} that the $\mathcal{O} \G^F$-$\mathcal{O} \Levi^F$-bimodule $H_c^d(\Y_\U,\mathcal{O}) e_s^{\Levi^F}$ extends to an $\mathcal{O} \G^F$-$\mathcal{O}\N^F$-bimodule. In the remaining cases Proposition \ref{ext} asserts that the bimodule $H_c^d(\Y_\U^\G,\mathcal{O}) e_s^{\Levi^F}$ extends to an $\mathcal{O}\G^F$-$\mathcal{O} \N^F$-bimodule.

By Proposition \ref{Morita} the extended bimodule induces a Morita equivalence between $\mathcal{O} \N^F e_s^{\Levi^F}$ and $\mathcal{O} \G^F e_s^{\G^F}$. In particular, one observes that the conclusion of \cite[Theorem 7.6]{Dat} and therefore of \cite[Theorem 7.7]{Dat} holds true in this case. This proves Theorem \ref{final}.
\end{proof}

\printindex

\bibliographystyle{plain}

\end{document}